\documentclass[11pt]{amsart}
\usepackage[margin=30mm]{geometry}
\usepackage{amsmath,amssymb}
\usepackage{amsthm}
\usepackage{mathrsfs}

\newtheorem{thm}{Theorem}[section]

\newtheorem{lem}{Lemma}[section]
\newtheorem{cor}{Corollary}[section]
\newtheorem{defi}{Definition}[section]

\newtheorem{rem}{Remark}[section]

\newtheorem*{claim}{\it{Claim}}
\newtheorem*{theorem}{\it{Theorem}}

\usepackage{enumitem}

\begin{document}

\title{Some results on s-limit shadowing and Li--Yorke type chaos}
\author{Noriaki Kawaguchi}
\address{Department of Mathematical and Computing Science, School of Computing, Tokyo Institute of Technology, 2-12-1 Ookayama, Meguro-ku, Tokyo 152-8552 Japan}
\email{gknoriaki@gmail.com}

\begin{abstract}
In \cite{K4}, for any continuous self-map of a compact metric space with s-limit shadowing, by using a $G_\delta$-partition of the phase space, a global description of Li--Yorke type chaos (with respect  to several particular Furstenberg families) is obtained. In this paper, we complement the results in \cite{K4} by some results concerning the s-limit shadowing and general Furstenberg families. An example is also given to illustrate the results.
\end{abstract}

\maketitle

\markboth{NORIAKI KAWAGUCHI}{Some results on s-limit shadowing and Li--Yorke type chaos}

\section{Introduction}

{\em Chaos} is a central concept in the modern theory of dynamical systems. The notion of {\em Li--Yorke chaos} was introduced in \cite{LY2}. Let $X$ denote a compact metric space endowed with a metric $d$. For a continuous map $f\colon X\to X$, $(x,y)\in X^2$ is said to be a {\em scrambled pair} for $f$ if it is a proximal and not asymptotic pair for $f$, i.e.,
\[
\liminf_{i\to\infty}d(f^i(x),f^i(y))=0\quad\text{and}\quad\limsup_{i\to\infty}d(f^i(x),f^i(y))>0.
\]
A subset $S$ of $X$ is said to be a {\em scrambled set} for $f$ if every $(x,y)\in S^2$ with $x\ne y$ is a scrambled pair for $f$. We say that $f$ exhibits {\em Li--Yorke chaos} if there is an uncountable scrambled set for $f$.

{\em Chain components} are important objects for global understanding of dynamical systems \cite{C}. Given any continuous map $f\colon X\to X$, let $\mathcal{C}(f)$ denote the set of chain components for $f$. For every $x\in X$, we have
\[
\lim_{i\to\infty}d(f^i(x),C(x))=0
\]
for some $C(x)\in\mathcal{C}(f)$. For any $x,y\in X$, if $C(x)\ne C(y)$, then $(x,y)$ is a {\em distal} pair for $f$, i.e.,
\[
\liminf_{i\to\infty}d(f^i(x),f^i(y))>0,
\]
thus, every proximal pair $(x,y)\in X^2$ for $f$ satisfies $C(x)=C(y)$. This implies that every scrambled set $S$ for $f$ is contained in the {\em basin}
\[
W^s(C)=\{x\in X\colon\lim_{i\to\infty}d(f^i(x),C)=0\}
\]
of some $C\in C(f)$. As a consequence, Li--Yorke chaos is a phenomenon observed in the basin of each chain component.

This viewpoint can be made more precise by using a {\em chain proximal relation} introduced in \cite{RW,Shi}. Given any continuous map $f\in X\to X$, every
$C\in\mathcal{C}(f)$ admits a partition $\mathcal{D}(C)$ with respect to the chain proximal relation $\sim_C$:
\[
C=\bigsqcup_{D\in\mathcal{D}(C)}D. 
\]
In \cite{K4}, by considering a kind of stable set $V^s(D)$ for each $D\in\mathcal{D}(C)$, the author extended it to a partition of $W^s(C)$:
\[
W^s(C)=\bigsqcup_{D\in\mathcal{D}(C)}V^s(D),
\]
where $V^s(D)$, $D\in\mathcal{D}(C)$, are $G_\delta$-subsets of $X$. We obtain accordingly a partition of $X$:
\[
X=\bigsqcup_{C\in\mathcal{C}(f)}W^s(C)=\bigsqcup_{C\in\mathcal{C}(f)}\bigsqcup_{D\in\mathcal{D}(C)}V^s(D).
\]
Since every scrambled set $S$ for $f$ is contained in $V^s(D)$ for some $C\in\mathcal{C}(f)$ and $D\in\mathcal{D}(C)$, as in \cite{K4}, this partition can be used to give a global description of Li--Yorke (more generally, Li--Yorke type) chaos.

The notion of Li--Yorke chaos was generalized in \cite{X} for $n$-tuples, $n\ge2$, and was also generalized in \cite{TX, XLT} with respect to {\em Furstenberg families} (see also \cite{LY1} and references therein). We call them as {\em Li--Yorke type chaos}. Let
\[
\mathbb{N}_0=\{0\}\cup\mathbb{N}=\{0,1,2,\dots\}
\]
and let $\mathcal{F}\subset2^{\mathbb{N}_0}$. We say that
$\mathcal{F}$ is a {\em Furstenberg family} if the following conditions are satisfied
\begin{itemize}
\item (hereditary upward) For any $A,B\subset\mathbb{N}_0$, $A\in\mathcal{F}$ and $A\subset B$ implies $B\in\mathcal{F}$,
\item (proper) $\mathcal{F}\ne\emptyset$ and $\mathcal{F}\ne2^{\mathbb{N}_0}$.
\end{itemize}
A Furstenberg family $\mathcal{F}$ is said to be {\em full} if
\[
\{i\in A\colon i\ge n\}\in\mathcal{F}
\]
for all $A\in\mathcal{F}$ and $n\ge0$. We say that a Furstenberg family $\mathcal{F}$ is {\em translation invariant} if
 \[
\{i+n\colon i\in A\}\in\mathcal{F}
\]
and
\[
\{i\in\mathbb{N}_0\colon i+n\in A\}\in\mathcal{F}
\]
for all $A\in\mathcal{F}$ and $n\ge0$.

Let $f\colon X\to X$ be a continuous map and let $\mathcal{F},\mathcal{G}$ be Furstenberg families. For any $x_1,x_2,\dots, x_n\in X$, $n\ge2$, and $r>0$, let
\[
S_f(x_1,x_2,\dots,x_n;r)=\{i\in\mathbb{N}_0\colon\min_{1\le j<k\le n}d(f^i(x_j),f^i(x_k))>r\}
\] 
and
\[
T_f(x_1,x_2,\dots,x_n;r)=\{i\in\mathbb{N}_0\colon\max_{1\le j<k\le n}d(f^i(x_j),f^i(x_k))<r\}.
\]
For any $\delta>0$, we say that $(x_1,x_2,\dots,x_n)$ is an {\em $(\mathcal{F},\mathcal{G})$-$\delta$-scrambled} $n$-tuple for $f$ if
\[
S_f(x_1,x_2,\dots,x_n;\delta)\in\mathcal{F},
\]
and
\[
T_f(x_1,x_2,\dots,x_n;\epsilon)\in\mathcal{G}
\]
for all $\epsilon>0$. Let $Y$ be a non-empty subset $X$. For any $n\ge2$ and $\delta>0$, we say that $Y$ is {\em dense $(\mathcal{F},\mathcal{G})$-$n$-$\delta$-chaotic} (resp.\:{\em generic $(\mathcal{F},\mathcal{G})$-$n$-$\delta$-chaotic}) for $f$ if the set of $(\mathcal{F},\mathcal{G})$-$\delta$-scrambled $n$-tuples in $Y^n$ for $f$, i.e.,
\begin{align*}
&\{(x_1,x_2,\dots,x_n)\in Y^n\colon S_f(x_1,x_2,\dots,x_n;\delta)\in\mathcal{F}\}\\
&\cap\bigcap_{\epsilon>0}\{(x_1,x_2,\dots,x_n)\in Y^n\colon T_f(x_1,x_2,\dots,x_n;\epsilon)\in\mathcal{G}\}
\end{align*}
is a dense (resp.\:residual) subset of $Y^n$.

\begin{rem}
\normalfont
For any topological space $Z$, a subset $S$ of $Z$ is said to be {\em residual} if $S$ contains a countable intersection of dense open subsets of $Z$. If $Z$ is a complete metric space, then by Baire Category Theorem, any residual subset $S$ of $Z$ is dense in $Z$. For any topological space $Z$, a subset $S$ of $Z$ is said to be a {\em $G_\delta$-subset} of $Z$ if $S$ is a countable intersection of open subsets of $Z$. By, e.g., Theorem 24.12 of \cite{W}, we know that a subspace $S$ of a complete metric space $Z$ is completely metrizable if and only if $S$ is a $G_\delta$-subset of $Z$.
\end{rem}

\begin{rem}
\normalfont
If $Y$ is dense $(\mathcal{F},\mathcal{G})$-$n$-$\delta$-chaotic for $f$, and if
\begin{align*}
&\{(x_1,x_2,\dots,x_n)\in Y^n\colon S_f(x_1,x_2,\dots,x_n;\delta)\in\mathcal{F}\}\\
&\cap\bigcap_{\epsilon>0}\{(x_1,x_2,\dots,x_n)\in Y^n\colon T_f(x_1,x_2,\dots,x_n;\epsilon)\in\mathcal{G}\}
\end{align*}
is a $G_\delta$-subset of $Y^n$, then $Y$ is generic $(\mathcal{F},\mathcal{G})$-$n$-$\delta$-chaotic for $f$. By Remark 1.1, if $Y$ is a $G_\delta$-subset of $X$ and generic $(\mathcal{F},\mathcal{G})$-$n$-$\delta$-chaotic for $f$, then $Y$ is dense $(\mathcal{F},\mathcal{G})$-$n$-$\delta$-chaotic for $f$.
\end{rem}

We recall a simplified version of a theorem of Mycielski \cite{M}. A topological space $Z$ is said to be {\em perfect} if $Z$ has no isolated point. For any complete metric space $Z$, a subset $S$ of $Z$ is said to be a {\em Mycielski set} if $S$ is a union of countably many Cantor sets (see \cite{BGKM}). Note that for any Mycielski set $S$ in $Z$ and an open subset $U$ of $Z$ with $S\cap U\ne\emptyset$, $S\cap U$ is an uncountable set.  

\begin{theorem}[Mycielski]
Let $Z$ be a perfect complete separable metric space. Given any $n\ge2$ and a residual subset $R_n$ of $Z^n$, there is a Mycielski set $S$ which is dense in $Z$ and satisfies $(x_1,x_2,\dots,x_n)\in R_n$ for all distinct $x_1,x_2,\dots,x_n\in S$. Moreover, if $R_n$ is a residual subset of $Z^n$ for each $n\ge2$, then there is a Mycielski set $S$ which is dense in $Z$ and satisfies $(x_1,x_2,\dots,x_n)\in R_n$ for all $n\ge2$ and distinct $x_1,x_2,\dots,x_n\in S$.
\end{theorem}

By this theorem, if $Y$ is a $G_\delta$-subset of $X$ and generic $(\mathcal{F},\mathcal{G})$-$n$-$\delta$-chaotic for $f$, then there is a dense Mycielski set $S$ in $Y$ such that 
$(x_1,x_2,\dots,x_n)$ is $(\mathcal{F},\mathcal{G})$-$\delta$-scrambled for $f$ for all distinct $x_1,x_2,\dots,x_n\in S$. If $Y$ is perfect, this condition in turn implies that $Y$ is dense $(\mathcal{F},\mathcal{G})$-$n$-$\delta$-chaotic for $f$.

The partition
\[
X=\bigsqcup_{C\in\mathcal{C}(f)}W^s(C)=\bigsqcup_{C\in\mathcal{C}(f)}\bigsqcup_{D\in\mathcal{D}(C)}V^s(D)
\]
mentioned above coarsely describes a global structure of {\em pseudo-orbits} which is compatible with the definition of Li--Yorke type chaos. By combining it with the so-called {\em shadowing}, we can obtain a precise description of the true orbit structure. Following \cite{K4}, we recall the definition of {\em s-limit shadowing}. 

\begin{defi}
\normalfont
Let $f\colon X\to X$ be a continuous map and let $\xi=(x_i)_{i\ge0}$ be a sequence of points in $X$. For $\delta>0$, $\xi$ is called a {\em $\delta$-limit-pseudo orbit} of $f$ if $d(f(x_i),x_{i+1})\le\delta$ for all $i\ge0$, and
\[
\lim_{i\to\infty}d(f(x_i),x_{i+1})=0.
\]
For $\epsilon>0$, $\xi$ is said to be {\em $\epsilon$-limit shadowed} by $x\in X$ if $d(f^i(x),x_i)\leq \epsilon$ for all $i\ge 0$, and
\[
\lim_{i\to\infty}d(f^i(x),x_i)=0.
\]
We say that $f$ has the {\em s-limit shadowing property} if for any $\epsilon>0$, there is $\delta>0$ such that every $\delta$-limit-pseudo orbit of $f$ is $\epsilon$-limit shadowed by some point of $X$.
\end{defi}

\begin{rem}
\normalfont
Let $f\colon X\to X$ be a continuous map and let $\xi=(x_i)_{i\ge0}$ be a sequence of points in $X$.

\begin{itemize}
\item For $\delta>0$, $\xi$ is called a {\em $\delta$-pseudo orbit} of $f$ if $d(f(x_i),x_{i+1})\le\delta$ for all $i\ge0$. For $\epsilon>0$, $\xi$ is said to be {\em $\epsilon$-shadowed} by $x\in X$ if $d(f^i(x),x_i)\leq \epsilon$ for all $i\ge 0$. We say that $f$ has the {\em shadowing property} if for any $\epsilon>0$, there is $\delta>0$ such that every $\delta$-pseudo orbit of $f$ is $\epsilon$-shadowed by some point of $X$.
\item $\xi$ is called a {\em limit-pseudo orbit} of $f$ if
\[
\lim_{i\to\infty}d(f(x_i),x_{i+1})=0,
\]
and said to be {\em limit shadowed} by $x\in X$
if
\[
\lim_{i\to\infty}d(f^i(x),x_i)=0.
\]
We say that $f$ has the {\em limit shadowing property} if every limit-pseudo orbit of $f$ is limit shadowed by some point of $X$.
\end{itemize}

It is not difficult to show that if $f$ has the s-limit shadowing property, then $f$ satisfies the shadowing property. It is also known that if $f$ has the s-limit shadowing property, then $f$ satisfies the limit shadowing property (see \cite{BGO}).  
\end{rem}

\begin{rem}
\normalfont
As far as the author knows, the name ``s-limit shadowing'' was introduced by Sakai \cite{Sa}, but it was implicitly considered by Bowen \cite{B}. In \cite{MO}, it is proved that the s-limit shadowing is dense in the space of all continuous self-maps of a compact topological manifold possibly with boundary. In a recent paper \cite{BCOT}, it is shown that the s-limit shadowing is generic in the space of all continuous circle maps.
\end{rem}

In \cite{K4}, the author considered Li--Yorke type chaos corresponding to several particular Furstenberg families. By assuming the s-limit shadowing and using the partition
\[
X=\bigsqcup_{C\in\mathcal{C}(f)}W^s(C)=\bigsqcup_{C\in\mathcal{C}(f)}\bigsqcup_{D\in\mathcal{D}(C)}V^s(D),
\]
a global description of Li--Yorke type chaos was obtained. We define the partition in Section 2. The purpose of this paper is to complement the results of \cite{K4} in a more general setting. The main results of this paper, which are presented below, concern the s-limit shadowing and general Furstenberg families. 

First, we recall a lemma from Section 4 of \cite{K4}.

\begin{lem}
Let $f\colon X\to X$ be a continuous map. Let $C\in\mathcal{C}(f)$ and $D\in\mathcal{D}(C)$. If $f$ has the s-limit shadowing property, then for any $x,y\in V^s(D)$ and $\epsilon>0$, there is $z\in V^s(D)$ such that $d(y,z)\le\epsilon$ and
\[
\lim_{i\to\infty}d(f^i(x),f^i(z))=0.
\]
\end{lem}

The first two results are, in fact, direct consequences of Lemma 1.1.

\begin{thm}
Let $f\colon X \to X$ be a continuous map. Let $C\in\mathcal{C}(f)$ and $D\in\mathcal{D}(C)$. Let $\mathcal{G}$ be a full Furstenberg family. If $f$ has the s-limit shadowing property, then for any $n\ge2$,
\[
B=\bigcap_{\epsilon>0}\{(x_1,x_2,\dots,x_n)\in[V^s(D)]^n\colon T_f(x_1,x_2,\dots,x_n;\epsilon)\in\mathcal{G}\}
\]
is a dense subset of $[V^s(D)]^n$.  
\end{thm}

\begin{thm}
Let $f\colon X \to X$ be a continuous map. Let $C\in\mathcal{C}(f)$ and $D\in\mathcal{D}(C)$. Let $\mathcal{F},\mathcal{G}$ be full Furstenberg families. For any $n\ge2$ and $\delta>0$, if $f$ has the s-limit shadowing property, and if there is an $(\mathcal{F},\mathcal{G})$-$\delta$-scrambled $n$-tuple
\[
(a_1,a_2,\dots,a_n)\in[V^s(D)]^n
\]
for $f$, then for every $0<r<\delta$, $V^s(D)$ is dense $(\mathcal{F},\mathcal{G})$-$n$-$r$-chaotic for $f$.  
\end{thm}

The next theorem concerns the topological conjugacy between two chain components. For two compact metric spaces $X,Y$ and continuous maps $f\colon X\to X, g\colon Y\to Y$, we say that a homeomorphism $h\colon X\to Y$ is a {\em topological conjugacy} if $h\circ f=g\circ h$ and we denote it as
\[
h\colon(X,f)\to(Y,g).
\]

\begin{thm}
Let $f\colon X \to X$, $g\colon Y\to Y$ be continuous maps where $X,Y$ are compact metric spaces. Let $C\in\mathcal{C}(f)$, $D\in\mathcal{D}(C)$, $C'\in\mathcal{C}(g)$, and let
\[
h\colon(C,f|_C)\to(C',g|_{C'})
\]
be a topological conjugacy. Note that $h(D)\in\mathcal{D}(C')$. Let $\mathcal{F},\mathcal{G}$ be full Furstenberg families. For any $n\ge2$ and $\delta>0$, if $g$ has the limit shadowing property, and if there is an $(\mathcal{F},\mathcal{G})$-$\delta$-scrambled $n$-tuple
\[
(a_1,a_2,\dots,a_n)\in[V^s(D)]^n
\]
for $f$, then  there are $\delta'>0$ and an $(\mathcal{F},\mathcal{G})$-$\delta'$-scrambled $n$-tuple
\[
(b_1,b_2,\dots,b_n)\in[V^s(h(D))]^n
\]
for $g$.
\end{thm}

By Theorems 1.2 and 1.3, we obtain the following corollary.

\begin{cor}
Let $f\colon X \to X$, $g\colon Y\to Y$ be continuous maps where $X,Y$ are compact metric spaces. Let $C\in\mathcal{C}(f)$, $D\in\mathcal{D}(C)$, $C'\in\mathcal{C}(g)$, and let
\[
h\colon(C,f|_C)\to(C',g|_{C'})
\]
be a topological conjugacy. Note that $h(D)\in\mathcal{D}(C')$. Let $\mathcal{F},\mathcal{G}$ be full Furstenberg families. For any $n\ge2$ and $\delta>0$, if $g$ has the s-limit shadowing property, and if $V^s(D)$ is dense $(\mathcal{F},\mathcal{G})$-$n$-$\delta$-chaotic for $f$, then there is $\delta'>0$ such that $V^s(h(D))$ is dense $(\mathcal{F},\mathcal{G})$-$n$-$\delta'$-chaotic for $g$.
\end{cor}

\begin{rem}
\normalfont
Theorem 1.3 and Corollary 1.1 indicate that, if $f\colon X\to X$ has the s-limit shadowing property, then  Li--Yorke type chaotic structure in the basin
\[
W^s(C)=\bigsqcup_{D\in\mathcal{D}(C)}V^s(D)
\]
of $C\in\mathcal{C}(f)$ depends only on the conjugacy class of $(C,f|_C)$. Note that the condition (1) in each of the three theorems of \cite{K4} specifies a conjugacy invariant class of $(C,f|_C)$.
\end{rem}

Let $f\colon X\to X$ be a continuous map. Following \cite{XLT}, we say that a Furstenberg family $\mathcal{F}$ is {\em compatible} with $(X,f)$ if for any open subset $U$ of $X$,
\[
\{x\in X\colon\{i\in\mathbb{N}_0\colon f^i(x)\in U\}\in\mathcal{F}\}
\]
is a $G_\delta$-subset of $X$. For any $n\ge2$, we denote by $f^{\times n}\colon X^n\to X^n$ the direct product of $n$-copies of  $f\colon X\to X$. For any $r>0$, let
\[
\Delta_r(n)=\{(x_1,x_2,\dots,x_n)\in X^n\colon\min_{1\le j<k\le n}d(x_j,x_k)>r\}
\]
and
\[
\Delta^r(n)=\{(x_1,x_2,\dots,x_n)\in X^n\colon\max_{1\le j<k\le n}d(x_j,x_k)<r\},
\]
which are open subsets of $X^n$. We see that
\begin{align*}
S_f(x_1,x_2,\dots,x_n;r)&=\{i\in\mathbb{N}_0\colon\min_{1\le j<k\le n}d(f^i(x_j),f^i(x_k))>r\}\\
&=\{i\in\mathbb{N}_0\colon(f^{\times n})^i(x_1,x_2,\dots,x_n)\in\Delta_r(n)\}
\end{align*}
and
\begin{align*}
T_f(x_1,x_2,\dots,x_n;r)&=\{i\in\mathbb{N}_0\colon\max_{1\le j<k\le n}d(f^i(x_j),f^i(x_k))<r\}\\
&=\{i\in\mathbb{N}_0\colon(f^{\times n})^i(x_1,x_2,\dots,x_n)\in\Delta^r(n)\};
\end{align*}
therefore, if $\mathcal{F}$ is compatible with $(X^n,f^{\times n})$, then for any non-empty subset $Y$ of $X$,
\[
\{(x_1,x_2,\dots,x_n)\in Y^n\colon S_f(x_1,x_2,\dots,x_n;\delta)\in\mathcal{F}\},
\]
$\delta>0$, and
\[
\bigcap_{\epsilon>0}\{(x_1,x_2,\dots,x_n)\in Y^n\colon T_f(x_1,x_2,\dots,x_n;\epsilon)\in\mathcal{G}\}
\]
are $G_\delta$-subsets of $Y^n$.

The next theorem provides a sufficient condition for generic $(\mathcal{F},\mathcal{G})$-$n$-$r$-chaos, $n\ge2$, $0<r<\delta$, under the assumption of the s-limit shadowing and a few moderate assumptions on $\mathcal{F},\mathcal{G}$. It generalizes the implication $(3)\implies(2)$ in the three theorems of \cite{K4}.
 
\begin{thm}
Let $f\colon X \to X$ be a continuous map. Let $C\in\mathcal{C}(f)$ and $D\in\mathcal{D}(C)$. Let $\mathcal{F},\mathcal{G}$ be full Furstenberg families compatible with
\[
(X^n,f^{\times n})
\]
for all $n\ge2$. For any $n\ge2$ and $\delta>0$, if $f$ has the s-limit shadowing property, $\mathcal{F}$ is translation invariant, and if there is an $n$-tuple
\[
(a_1,a_2,\dots,a_n)\in[V^s(D)]^n
\]
such that
\[
S_f(a_1,a_2,\dots,a_n;\delta)\in\mathcal{F},
\]
then for any $E\in\mathcal{D}(C)$ and $0<r<\delta$, $V^s(E)$ is generic $(\mathcal{F},\mathcal{G})$-$n$-$r$-chaotic for $f$.
\end{thm}

For a continuous map $f\colon X\to X$, we say that $(x_1,x_2,\dots,x_n)\in X^n$ is a {\em $\delta$-distal} $n$-tuple for $f$
if
\[
\inf_{i\ge0}\min_{1\le j<k\le n}d(f^i(x_j),f^i(x_k))>\delta.
\]
Note that if $(x_1,x_2,\dots,x_n)\in X^n$ is a {\em $\delta$-distal} $n$-tuple for $f$, then
\[
S_f(x_1,x_2,\dots,x_n;\delta)=\mathbb{N}_0\in\mathcal{F}
\]
for every Furstenberg family $\mathcal{F}$. By Theorem 1.4, we obtain the following corollary. 

\begin{cor}
Let $f\colon X \to X$ be a continuous map. Let $C\in\mathcal{C}(f)$ and $D\in\mathcal{D}(C)$. Let $\mathcal{F},\mathcal{G}$ be full Furstenberg families compatible with
\[
(X^n,f^{\times n})
\]
for all $n\ge2$. For any $n\ge2$ and $\delta>0$, if $f$ has the s-limit shadowing property, $\mathcal{F}$ is translation invariant, and if there is a $\delta$-distal $n$-tuple
\[
(a_1,a_2,\dots,a_n)\in D^n
\]
for $f$, then for any $E\in\mathcal{D}(C)$ and $0<r<\delta$, $V^s(E)$ is generic $(\mathcal{F},\mathcal{G})$-$n$-$r$-chaotic for $f$.
\end{cor}

This paper consists of four sections. In Section 2, we give the definition of the $G_\delta$-partition mentioned above. In Section 3, we prove the main theorems. In Section 4, we give an example concerning Theorem 1.4.

\section{The $G_\delta$-partition}

In this section, following \cite{K4}, we give the precise definition of the partition
\[
X=\bigsqcup_{C\in\mathcal{C}(f)}W^s(C)=\bigsqcup_{C\in\mathcal{C}(f)}\bigsqcup_{D\in\mathcal{D}(C)}V^s(D)
\]
for any continuous map $f\colon X\to X$. In what follows $X$ denotes a compact metric space endowed with a metric $d$. We begin with a definition.

\begin{defi}
\normalfont
Given a continuous map $f\colon X\to X$ and $\delta>0$, a finite sequence $(x_i)_{i=0}^{k}$ of points in $X$, where $k>0$ is a positive integer, is called a {\em $\delta$-chain} of $f$ if $d(f(x_i),x_{i+1})\le\delta$ for every $0\le i\le k-1$.  A $\delta$-chain $(x_i)_{i=0}^{k}$ of $f$ with $x_0=x_k$ is said to be a {\em $\delta$-cycle} of $f$. 
\end{defi}

Let $f\colon X\to X$ be a continuous map. For any $x,y\in X$ and $\delta>0$, the notation $x\rightarrow_\delta y$ means that there is a $\delta$-chain $(x_i)_{i=0}^k$ of $f$ with $x_0=x$ and $x_k=y$. We write $x\rightarrow y$ if $x\rightarrow_\delta y$ for all $\delta>0$. We say that $x\in X$ is a {\em chain recurrent point} for $f$ if $x\rightarrow x$, or equivalently, for any $\delta>0$, there is a $\delta$-cycle $(x_i)_{i=0}^{k}$ of $f$ with $x_0=x_k=x$. Let $CR(f)$ denote the set of chain recurrent points for $f$. We define a relation $\leftrightarrow$ in
\[
CR(f)^2=CR(f)\times CR(f)
\]
by: for any $x,y\in CR(f)$, $x\leftrightarrow y$ if and only if $x\rightarrow y$ and $y\rightarrow x$. Note that $\leftrightarrow$ is a closed equivalence relation in $CR(f)^2$ and satisfies $x\leftrightarrow f(x)$ for all $x\in CR(f)$. An equivalence class $C$ of $\leftrightarrow$ is called a {\em chain component} for $f$. We denote by $\mathcal{C}(f)$ the set of chain components for $f$.

A subset $S$ of $X$ is said to be $f$-invariant if $f(S)\subset S$. For an $f$-invariant subset $S$ of $X$, we say that $f|_S\colon S\to S$ is {\em chain transitive} if for any $x,y\in S$ and $\delta>0$, there is a $\delta$-chain $(x_i)_{i=0}^k$ of $f|_S$ with $x_0=x$ and $x_k=y$.

\begin{rem}
\normalfont
The following properties hold
\begin{itemize}
\item $CR(f)=\bigsqcup_{C\in\mathcal{C}(f)}C$,
\item every $C\in\mathcal{C}(f)$ is a closed $f$-invariant subset of $CR(f)$,
\item $f|_C\colon C\to C$ is chain transitive for all $C\in\mathcal{C}(f)$,
\item for any $f$-invariant subset $S$ of $X$, if $f|_S\colon S\to S$ is chain transitive, then $S\subset C$ for some $C\in\mathcal{C}(f)$.
\end{itemize}
\end{rem}

\begin{rem}
\normalfont
We recall the so-called fundamental theorem of dynamical systems by Conley \cite{C} (see also \cite{H}). Given any continuous map $f\colon X\to X$, we say that a continuous function $\Lambda\colon X\to\mathbb{R}$ is a {\em complete Lyapunov function} for $f$ if the following conditions are satisfied
\begin{itemize}
\item $\Lambda(f(x))<\Lambda(x)$ for every $x\in X\setminus CR(f)$,
\item for any $x,y\in CR(f)$,  $\Lambda(x)=\Lambda(y)$ if and only if $x\leftrightarrow y$,
\item $\Lambda(CR(f))$ is a nowhere dense subset of $\mathbb{R}$.
\end{itemize}

\begin{theorem}
There is a complete Lyapunov function $\Lambda\colon X\to\mathbb{R}$ for $f$.
\end{theorem}
\end{rem}

Let $C\in\mathcal{C}(f)$ and let $g=f|_C\colon C\to C$. Given any $\delta>0$, the {\em length} of a $\delta$-cycle $(x_i)_{i=0}^k$ of $g$ is defined to be $k$. Let $m=m(C,\delta)>0$ be the greatest common divisor of the lengths of all $\delta$-cycles of $g$. A relation $\sim_{C,\delta}$ in $C^2$ is defined by: for any $x,y\in C$, $x\sim_{C,\delta}y$ if and only if there is a $\delta$-chain $(x_i)_{i=0}^k$ of $g$ with $x_0=x$, $x_k=y$, and $m|k$. 

\begin{rem}
\normalfont
The following properties hold 
\begin{itemize}
\item[(P1)] $\sim_{C,\delta}$ is an open and closed $(g\times g)$-invariant equivalence relation in $C^2=C\times C$,
\item[(P2)] any $x,y\in C$ with $d(x,y)\le\delta$ satisfies $x\sim_{C,\delta}y$, so for every $\delta$-chain $(x_i)_{i=0}^k$ of $g$, we have $g(x_i)\sim_{C,\delta}x_{i+1}$ for each $0\le i\le k-1$, implying $x_i\sim_{C,\delta}g^i(x_0)$ for every $0\le i\le k$,
\item[(P3)] for any $x\in C$ and $n\ge0$, $x\sim_{C,\delta}g^{mn}(x)$,
\item[(P4)] there exists $N>0$ such that for any $x,y\in C$ with $x\sim_{C,\delta}y$ and $n\ge N$, there is a $\delta$-chain $(x_i)_{i=0}^k$ of $g$ with $x_0=x$, $x_k=y$, and $k=mn$.
\end{itemize}
A proof of property (P2) can be found in \cite{K1}. Property (P4) is due to Lemma 2.3 of \cite{BMR} and in fact a consequence of a Schur's theorem implying that for any positive integers $m,k_1,k_2,\dots,k_n$, if
\[
\gcd(k_1,k_2,\dots,k_n)=m,
\]
then every sufficiently large multiple $M$ of $m$ can be expressed as a linear combination
\[
M=a_1k_1+a_2k_2+\cdots+a_nk_n
\]
where $a_1,a_2,\dots,a_n$ are non-negative integers. 
\end{rem}

Fix $x\in C$ and let $D_i$, $i\ge0$, denote the equivalence class of $\sim_{C,\delta}$ including $g^i(x)$. Then, $D_m=D_0$, and
\[
C=\bigsqcup_{i=0}^{m-1}D_i
\]
gives the partition of $C$ into the equivalence classes of $\sim_{C,\delta}$. Note that every $D_i$, $0\le i\le m-1$, is an open and closed subset of $C$ and satisfies $g(D_i)=D_{i+1}$. We call
\[
\mathcal{D}(C,\delta)=\{D_i\colon 0\le i\le m-1\}
\]
the {\em $\delta$-cyclic decomposition} of $C$.

\begin{defi}
\normalfont
We define a relation $\sim_C$ in $C^2$ by: for any $x,y\in C$, $x\sim_C y$ if and only if $x\sim_{C,\delta}y$ for all $\delta>0$, which is a closed $(g\times g)$-invariant equivalence relation in $C^2$. We denote by $\mathcal{D}(C)$ the set of equivalence classes of $\sim_C$.
\end{defi}

\begin{rem}
\normalfont
The relation $\sim_C$ was introduced in \cite{Shi} and rediscovered in \cite{RW} (based on the argument given in \cite[Exercise 8.22]{A}). We say that $(x,y)\in C^2$ is a {\em chain proximal pair} for $g$ if for any $\delta>0$, there is a pair of $\delta$-chains
\[
((x_i)_{i=0}^k,(y_i)_{i=0}^k)
\]
of $g$ with $(x_0,y_0)=(x,y)$ and $x_k=y_k$. As stated in Remark 8 of \cite{RW}, it holds that for any $x,y\in C$, $x\sim_C y$ if and only if $(x,y)$ is a chain proximal pair for $g$. By this equivalence, we say that $\sim_C$ is a {\em chain proximal relation} for $g$.
\end{rem}

\begin{rem}
\normalfont
Let $X,Y$ be compact metric spaces and let $f\colon X\to X$, $g\colon Y\to Y$ be continuous maps. Let $C\in\mathcal{C}(f),C'\in\mathcal{C}(g)$ and let
\[
h\colon(C,f|_C)\to(C',g|_{C'})
\]
be a topological conjugacy. For any $x,y\in C$, $(x,y)$ is a chain proximal pair for $f|_C$ if and only if $(h(x),h(y))$ is a chain proximal pair for $g|_{C'}$, thus $x\sim_C y$ if and only if $h(x)\sim_{C'}h(y)$. By this, we observe that $h(D)\in\mathcal{D}(C')$ for every $D\in\mathcal{D}(C)$ and that $\hat{h}\colon\mathcal{D}(C)\to\mathcal{D}(C')$ defined by
$\hat{h}(D)=h(D)$ for all $D\in\mathcal{D}(C)$ is a bijection. 
\end{rem}

Given any continuous map $f\colon X\to X$, let
\[
W^s(C)=\{x\in X\colon\lim_{i\to\infty}d(f^i(x),C)=0\}
\]
for all $C\in\mathcal{C}(f)$.  Note that $C\subset W^s(C)$ for every $C\in\mathcal{C}(f)$. A proof of the following lemma is given in \cite{K4}.
 
\begin{lem}
We have the following properties
\begin{itemize}
\item
\[
X=\bigsqcup_{C\in\mathcal{C}(f)}W^s(C),
\]
\item $W^s(C)$ is a $G_\delta$-subset of $X$ for every $C\in\mathcal{C}(f)$.
\end{itemize}
\end{lem}

Let $C\in\mathcal{C}(f)$ and let $D\in\mathcal{D}(C)$. By the definition of $\sim_C$, for any $\delta>0$, there is unique $D_\delta\in\mathcal{D}(C,\delta)$ such that $D\subset D_\delta$. Then, $D_{\delta_1}\subset D_{\delta_2}$ for all $0<\delta_1<\delta_2$, and
\[
D=\bigcap_{\delta>0}D_\delta.
\]
Let
\[
V^s(D)=\bigcap_{\delta>0}\{x\in W^s(C)\colon\lim_{i\to\infty}d(f^i(x),f^i(D_{\delta}))=0\},
\]
$D\in\mathcal{D}(C)$, and note that $D\subset V^s(D)$ for every $D\in\mathcal{D}(C)$. A proof of the following lemma is also given in \cite{K4}.

\begin{lem}
We have the following properties
\begin{itemize}
\item
\[
W^s(C)=\bigsqcup_{D\in\mathcal{D}(C)}V^s(D),
\]
\item $V^s(D)$ is a $G_\delta$-subset of $W^s(C)$ for every $D\in\mathcal{D}(C)$.
\end{itemize}
\end{lem}

\begin{rem}
\normalfont
We can show that for any $C\in\mathcal{C}(f)$ and $D\in\mathcal{D}(C)$,
\[
V^s(D)=\{x\in W^s(C)\colon\liminf_{i\to\infty}d(f^i(x),f^i(D))=0\}.
\]
\end{rem}

By the above two lemmas, we obtain
\[
X=\bigsqcup_{C\in\mathcal{C}(f)}W^s(C)=\bigsqcup_{C\in\mathcal{C}(f)}\bigsqcup_{D\in\mathcal{D}(C)}V^s(D).
\]
Moreover, $V^s(D)$ is a $G_\delta$-subset of $X$ for all $C\in\mathcal{C}(f)$ and $D\in\mathcal{D}(C)$.

\begin{rem}
\normalfont
Given a continuous map $f\colon X\to X$, let $C,C'\in\mathcal{C}(f)$ and $D\in\mathcal{D}(C),D'\in\mathcal{D}(C')$. For any $x,y\in X$, if $(x,y)$ is a proximal pair for $f$, i.e.,
\[
\liminf_{i\to\infty}d(f^i(x),f^i(y))=0,
\]
then $C=C'$ and $D=D'$. Let
\[
\mathcal{F}_{\rm infinite}=\{A\subset\mathbb{N}_0\colon |A|=\infty\},
\]
a Furstenberg family. For any $x_1,x_2,\dots,x_n\in X$, $n\ge2$, we have
\[
\liminf_{i\to\infty}\max_{1\le j<k\le n}d(f^i(x_j),f^i(x_k))=0,
\]
i.e., $(x_1,x_2,\dots,x_n)$ is a proximal $n$-tuple for $f$ if and only if
\[
T_f(x_1,x_2,\dots,x_n;\epsilon)\in\mathcal{F}_{\rm infinite}
\]
for all $\epsilon>0$. If a Furstenberg family $\mathcal{G}$ is full, then $\mathcal{G}\subset\mathcal{F}_{\rm infinite}$. It follows that for any $x_1,x_2,\dots,x_n\in X$, if $\mathcal{G}$ is a full Furstenberg family, and if
\[
T_f(x_1,x_2,\dots,x_n;\epsilon)\in\mathcal{G}
\]
for all $\epsilon>0$, then $x_1,x_2,\dots,x_n\in V^s(D)$ for some $C\in\mathcal{C}(f)$ and $D\in\mathcal{D}(C)$.  
\end{rem}

\section{Proof of Theorems}

In this section, we prove the main theorems. First, Theorems 1.1 and 1.2 are immediate consequences of Lemma 1.1 in Section 1.

\begin{proof}[Proof of Theorem 1.1]
Fix $(a_1,a_2,\dots,a_n)\in[V^s(D)]^n$. By Lemma 1.1,
\[
A=\{(x_1,x_2,\dots,x_n)\in[V^s(D)]^n\colon\lim_{i\to\infty}d(f^i(a_j),f^i(x_j))=0,\:1\le\forall j\le n\}
\]
is a dense subset of $[V^s(D)]^n$. Note that
\[
A\subset\{(x_1,x_2,\dots,x_n)\in[V^s(D)]^n\colon\lim_{i\to\infty}\max_{1\le j<k\le n}d(f^i(x_j),f^i(x_k))=0\}.
\]
Since $\mathcal{G}$ is full, we have
\[
\{(x_1,x_2,\dots,x_n)\in[V^s(D)]^n\colon\lim_{i\to\infty}\max_{1\le j<k\le n}d(f^i(x_j),f^i(x_k))=0\}\subset B.
\]
It follows that $B$ is a dense subset of $[V^s(D)]^n$, proving the theorem. 
\end{proof}

\begin{proof}[Proof of Theorem 1.2]
By Lemma 1.1,
\[
A=\{(x_1,x_2,\dots,x_n)\in[V^s(D)]^n\colon\lim_{i\to\infty}d(f^i(a_j),f^i(x_j))=0,\:1\le\forall j\le n\}
\]
is a dense subset of $[V^s(D)]^n$. Since $(a_1,a_2,\dots,a_n)$ is $(\mathcal{F},\mathcal{G})$-$\delta$-scrambled for $f$, and $\mathcal{F},\mathcal{G}$ are both full, every $(x_1,x_2,\dots,x_n)\in A$ is $(\mathcal{F},\mathcal{G})$-$r$-scrambled for $f$. It follows that $V^s(D)$ is dense $(\mathcal{F},\mathcal{G})$-$n$-$r$-chaotic for $f$, proving the theorem. 
\end{proof}

Next, we prove Theorem 1.3.  

\begin{proof}[Proof of Theorem 1.3]
For each $1\le j \le n$, since $a_j\in V^s(D)\subset W^s(C)$ and so
\[
\lim_{i\to\infty}d(f^i(a_j),C)=0,
\]
by taking $x^{(j)}_i\in C$, $i\ge0$, with $d(f^i(a_j),x^{(j)}_i)=d(f^i(a_j),C)$ for all $i\ge0$, we have
\[
\tag{A} \lim_{i\to\infty}d(f^i(a_j),x^{(j)}_i)=0
\]
and
\[
\lim_{i\to\infty}d(f(x^{(j)}_i),x^{(j)}_{i+1})=0.
\]
Fix $1\le j \le n$. Since $a_j\in V^s(D)$, we have
\[
\lim_{i\to\infty}d(f^i(a_j),f^i(D_\eta))=0
\]
and so
\[
\lim_{i\to\infty}d(x^{(j)}_i,f^i(D_\eta))=0
\]
for all $\eta>0$ and $D_\eta\in\mathcal{D}(C,\eta)$ with $D\subset D_\eta$. Let $d'$ denote the metric on $Y$. Letting
\[
\tag{B} y^{(j)}_i=h(x^{(j)}_i)
\]
for all $i\ge0$, we obtain
\[
\lim_{i\to\infty}d'(g(y^{(j)}_i),y^{(j)}_{i+1})=\lim_{i\to\infty}d'(g(h(x^{(j)}_i)),h(x^{(j)}_{i+1}))=\lim_{i\to\infty}d'(h(f(x^{(j)}_i)),h(x^{(j)}_{i+1}))=0,
\]
and
\[
\lim_{i\to\infty}d'(y^{(j)}_i,g^i(h(D_\eta)))=\lim_{i\to\infty}d'(h(x^{(j)}_i),h(f^i(D_\eta)))=0
\]
for all $\eta>0$ and $D_\eta\in\mathcal{D}(C,\eta)$ with $D\subset D_\eta$. Since $g$ has the limit shadowing property, we have
\[
\tag{C} \lim_{i\to\infty}d'(g^i(b_j),y^{(j)}_i)=0
\]
for some $b_j\in Y$. For every $\eta'>0$, by taking $\eta>0$ with $h(D_\eta)\subset D'_{\eta'}$, where $D'_{\eta'}\in\mathcal{D}(C',\eta')$ and $h(D)\subset D'_{\eta'}$, we obtain
\[
\lim_{i\to\infty}d'(g^i(b_j),g^i(D'_{\eta'}))=\lim_{i\to\infty}d'(g^i(b_j),g^i(h(D_\eta)))=\lim_{i\to\infty}d'(y^{(j)}_i,g^i(h(D_\eta)))=0.
\]
It follows that $b_j\in V^s(h(D))$. Since $(a_1,a_2,\dots,a_n)$ is $(\mathcal{F},\mathcal{G})$-$\delta$-scrambled for $f$, and $\mathcal{F},\mathcal{G}$ are both full, by (A), (B), (C), we conclude that $(b_1,b_2,\dots,b_n)\in[V^s(h(D))]^n$ is $(\mathcal{F},\mathcal{G})$-$\delta'$-scrambled for $g$ for some $\delta'>0$, completing the proof. 
\end{proof}

For the proof of Theorem 1.4, we need a lemma which may be of independent interest.

\begin{lem}
Let $f\colon X\to X$ be a continuous map. Let $C\in\mathcal{C}(f)$ and $D, E\in\mathcal{D}(C)$. If $f$ has the s-limit shadowing property, then for any $n\ge1$, $x_1,x_2,\dots,x_n\in V^s(D)$, $y_1,y_2,\dots,y_n\in\overline{V^s(E)}$, and $\epsilon>0$, there are $z_1,z_2,\dots,z_n\in\overline{V^s(E)}$ and $K,L\ge0$ such that
\[
d(y_j,z_j)\le\epsilon
\]
for all $1\le j\le n$, and
\[
d(f^{K+i}(x_j),f^{L+i}(z_j))\le\epsilon
\]
for all $1\le j\le n$ and $i\ge0$.
\end{lem}

\begin{proof}
Since $f$ has the s-limit shadowing property, there is $\delta>0$ such that every $2\delta$-limit-pseudo orbit of $f$ is $\epsilon/2$-limit shadowed by some point of $X$. For this $\delta>0$, we take $D_\delta, E_\delta\in\mathcal{D}(C,\delta)$ with $D\subset D_\delta$ and $E\subset E_\delta$. Let $m=|\mathcal{D}(C,\delta)|$ and note that $f^m(D_\delta)=D_\delta$, $f^m(E_\delta)=E_\delta$. By property (P4) of $\sim_{C,\delta}$, we can choose $N>0$ such that for any $u,v\in D_\delta$ and $n\ge N$, there is a $\delta$-chain $(z_i)_{i=0}^{mn}$ of $f$ with $z_0=u$ and $z_{mn}=v$. Since $x_1,x_2,\dots,x_n\in V^s(D)$, there is $M>0$ such that
\[
d(f^{mk}(x_j),D_\delta)=d(f^{mk}(x_j),f^{mk}(D_\delta))\le\delta
\]
for all $1\le j\le n$ and $k\ge M$. Since $y_1,y_2,\dots,y_n\in\overline{V^s(E)}$, there are $a_1,a_2,\dots,a_n\in V^s(E)$ such that
\[
d(y_j,a_j)\le\epsilon/2
\]
for all $1\le j\le n$. We take  $a>0$ with $f^a(E_\delta)=D_\delta$. Then, since $a_1,a_2,\dots,a_n\in V^s(E)$, there is $J>0$ such that
\[
d(f^{a+m\alpha}(a_j),D_\delta)\le\delta
\]
for all $1\le j\le n$ and $\alpha\ge J$. Fix $1\le j\le n$ and take $p\in D_\delta$, $q_k,r_k\in D_\delta$, $k\ge M$, such that the following properties hold
\begin{itemize}
\item $d(f^{a+mJ}(a_j),p)=d(f^{a+mJ}(a_j),D_\delta)\le\delta$,
\item $d(f^{mk}(x_j),q_k)=d(f^{mk}(x_j),D_\delta)\le\delta$ for all $k\ge M$,
\item $d(f^{a+mJ+m(k-M+2N)}(a_j),r_k)=d(f^{a+mJ+m(k-M+2N)}(a_j),D_\delta)\le\delta$ for all $k\ge M$.
\end{itemize}
Since $p,q_M\in D_\delta$, the choice of $N$ gives a $\delta$-chain $(z_i)_{i=0}^{mN}$ of $f$ with $z_0=p$ and $z_{mN}=q_M$. For every $k>M$, since $q_k,r_k\in D_\delta$, the choice of $N$ gives a $\delta$-chain $(w^{(k)}_i)_{i=0}^{mN}$ of $f$ with $w^{(k)}_0=q_k$ and $w^{(k)}_{mN}=r_k$. For each $k>M$, consider the $2\delta$-limit-pseudo orbit
\begin{align*}
\xi_k=&(a_j,f(a_j),\dots,f^{a+mJ-1}(a_j),z_0,z_1,\dots,z_{mN-1},f^{mM}(x_j),f^{mM+1}(x_j),\dots,f^{mk-1}(x_j),\\
&w^{(k)}_0,w^{(k)}_1,\dots,w^{(k)}_{mN-1},f^{a+mJ+m(k-M+2N)}(a_j),f^{a+mJ+m(k-M+2N)+1}(a_j),\dots)
\end{align*}
of $f$, which is $\epsilon/2$-limit shadowed by $w_k\in X$. For every $k>M$, by $a_j\in V^s(E)$ and  
\[
\lim_{i\to\infty}d(f^{a+mJ+m(k-M+2N)+i}(a_j),f^{a+mJ+m(k-M+2N)+i}(w_k))=0,
\]
we obtain $w_k\in V^s(E)$. Since
\[
d(f^{mM+i}(x_j),f^{a+mJ+mN+i}(w_k))\le\epsilon
\]
for all $k>M$ and $0\le i\le m(k-M)-1$, by taking a sequence $M<k_1<k_2<\cdots$ and $z_j\in X$ with
\[
\lim_{l\to\infty}w_{k_l}=z_j,
\]
we have $z_j\in\overline{V^s(E)}$, and
\[
d(f^{mM+i}(x_j),f^{a+mJ+mN+i}(z_j))\le\epsilon
\]
for all $i\ge0$. Note that $d(a_j,w_{k_l})\le\epsilon/2$ for all $l\ge1$. This implies $d(a_j,z_j)\le\epsilon/2$ and so $d(y_j,z_j)\le\epsilon$. Since $1\le j\le n$ is arbitrary, letting $K=mM$ and $L=a+mJ+mN$, we obtain
\[
d(f^{K+i}(x_j),f^{L+i}(z_j))\le\epsilon
\]
for all $1\le j\le n$ and $i\ge0$, completing the proof.
\end{proof}

Finally, we prove Theorem 1.4.

\begin{proof}[Proof of Theorem 1.4]
Since $\mathcal{F},\mathcal{G}$ are compatible with $(X^n,f^{\times n})$, for every non-empty subset $Y$ of $X$,
\[
\{(x_1,x_2,\dots,x_n)\in Y^n\colon S_f(x_1,x_2,\dots,x_n;r)\in\mathcal{F}\},
\]
$r>0$, and
\[
\bigcap_{\epsilon>0}\{(x_1,x_2,\dots,x_n)\in Y^n\colon T_f(x_1,x_2,\dots,x_n;\epsilon)\in\mathcal{G}\}
\]
are $G_\delta$-subsets of $Y^n$. Given any $E\in\mathcal{D}(C)$ and $0<r<\delta$, since $\mathcal{F}$ is full and translation invariant, and $(a_1,a_2,\dots,a_n)\in[V^s(D)]^n$ satisfies
\[
S_f(a_1,a_2,\dots,a_n;\delta)\in\mathcal{F},
\]
by Lemma 3.1, 
\[
S=\{(x_1,x_2,\dots,x_n)\in[\overline{V^s(E)}]^n\colon S_f(x_1,x_2,\dots,x_n;r)\in\mathcal{F}\}
\]
is a dense $G_\delta$-subset of $[\overline{V^s(E)}]^n$. Since $[V^s(E)]^n$ is a dense $G_\delta$-subset of $[\overline{V^s(E)}]^n$,
\[
A=S\cap[V^s(E)]^n=\{(x_1,x_2,\dots,x_n)\in[V^s(E)]^n\colon S_f(x_1,x_2,\dots,x_n;r)\in\mathcal{F}\}
\]
is a dense $G_\delta$-subset of $[\overline{V^s(E)}]^n$, implying that $A$ is a dense $G_\delta$-subset of $[V^s(E)]^n$. On the other hand, since $\mathcal{G}$ is full, by Theorem 1.1,
\[
B=\bigcap_{\epsilon>0}\{(x_1,x_2,\dots,x_n)\in[V^s(E)]^n\colon T_f(x_1,x_2,\dots,x_n;\epsilon)\in\mathcal{G}\}
\]
is a dense $G_\delta$-subset of $[V^s(E)]^n$. We conclude that $A\cap B$ is a dense $G_\delta$-subset of $[V^s(E)]^n$, thus $V^s(E)$ is generic $(\mathcal{F},\mathcal{G})$-$n$-$r$-chaotic for $f$, proving the theorem.
\end{proof}

\section{Example}

In this section, we give an example concerning Theorem 1.4 in Section 1. We need a lemma. Let $X$ be a compact metric space endowed with a metric $d$. For any $\alpha>1$, we define a metric $D$ on $X^\mathbb{N}$ by
\[
D(x,y)=\sup_{n\ge1}\alpha^{-n}d(x_n,y_n)
\]
for all $x=(x_n)_{n\ge1},y=(y_n)_{n\ge1}\in X^\mathbb{N}$. The shift map $\sigma\colon X^\mathbb{N}\to X^\mathbb{N}$ is defined by for any $x=(x_n)_{n\ge1},y=(y_n)_{n\ge1}\in X^\mathbb{N}$, $y=\sigma(x)$ if only if $y_n=x_{n+1}$ for all $n\ge1$.

\begin{lem}
Let $\xi=(x^{(i)})_{i\ge0}$ be a sequence of points in $X^\mathbb{N}$. Let $x_n=x_1^{(n-1)}$ for every $n\ge1$, and let $x=(x_n)_{n\ge1}\in X^\mathbb{N}$.
\begin{itemize}
\item If $\xi$ is a limit-pseudo orbit of $\sigma$, then $\xi$ is limit shadowed by $x$.
\item For any $\delta>0$, if  $\xi$ is a $\delta$-limit-pseudo orbit of $\sigma$, then $\xi$ is $\delta/(\alpha-1)$-limit shadowed by $x$.
\end{itemize}
\end{lem}

\begin{proof}
Let $\epsilon_l=D(\sigma(x^{(l)}),x^{(l+1)})$ for all $l\ge0$, and let $\beta_i=\sup_{l\ge i}\epsilon_l$ for all $i\ge0$. Given any $i\ge0$ and $n\ge1$, we have
\begin{align*}
d(\sigma^i(x)_n,x_n^{(i)})&=d(x_{i+n},x_n^{(i)})=d(x_1^{(i+n-1)},x_n^{(i)})\\
&\le\sum_{j=1}^{n-1}d(x_j^{(i+n-j)},x_{j+1}^{(i+n-j-1)})=\sum_{j=1}^{n-1}d(x_j^{(i+n-j)},\sigma(x^{(i+n-j-1)})_j)\\
&\le\sum_{j=1}^{n-1}\alpha^j D(x^{(i+n-j)},\sigma(x^{(i+n-j-1)}))=\sum_{j=1}^{n-1}\alpha^j\epsilon_{i+n-j-1}.
\end{align*}
It follows that for any $i\ge0$ and $n\ge1$,
\begin{align*}
\alpha^{-n}d(\sigma^i(x)_n,x_n^{(i)})&\le\sum_{j=1}^{n-1}\alpha^{-n+j}\epsilon_{i+n-j-1}=\sum_{k=1}^{n-1}\alpha^{-k}\epsilon_{i+k-1}\\
&\le\sum_{k=1}^{n-1}\alpha^{-k}\beta_i\le\sum_{k=1}^\infty\alpha^{-k}\beta_i=\beta_i/(\alpha-1).
\end{align*}
We obtain
\[
D(\sigma^i(x),x^{(i)})=\sup_{n\ge1}\alpha^{-n}d(\sigma^i(x)_n,x_n^{(i)})\le\beta_i/(\alpha-1)
\]
for all $i\ge0$. If $\xi$ is a limit-pseudo orbit of $\sigma$, then since
\[
\lim_{i\to\infty}D(\sigma^i(x),x^{(i)})=\lim_{i\to\infty}\beta_i=0,
\]
$\xi$ is limit shadowed by $x$. Moreover, if $\xi$ is a $\delta$-limit-pseudo orbit of $\sigma$, then
we have
\[
D(\sigma^i(x),x^{(i)})\le\beta_i/(\alpha-1)\le\beta_0/(\alpha-1)\le\delta/(\alpha-1)
\]
for all $i\ge0$, thus $\xi$ is $\delta/(\alpha-1)$-limit shadowed by $x$. This completes the proof.
\end{proof}

\begin{rem}
\normalfont
Let $S$ be a finite set endowed with a metric $d$. For any $\alpha>1$, we define a metric $D$ on $S^\mathbb{N}$ by
\[
D(x,y)=\sup_{n\ge1}\alpha^{-n}d(x_n,y_n)
\]
for all $x=(x_n)_{n\ge1},y=(y_n)_{n\ge1}\in S^\mathbb{N}$. Let $\sigma\colon S^\mathbb{N}\to S^\mathbb{N}$ be the shift map. A closed $\sigma$-invariant subset $Y$ of $S^\mathbb{N}$ is called a {\em subshift}. A subshift $Y$ of $S^\mathbb{N}$ is called a {\em subshift of finite type} if there are $N\ge1$ and $F\subset S^{N+1}$ such that for any $x=(x_n)_{n\ge1}\in S^\mathbb{N}$, $x\in Y$ if and only if $(x_n,x_{n+1},\dots,x_{n+N})\in F$ for all $n\ge1$. If $Y$ is a subshift of finite type, then for any $\epsilon>0$, there is $\delta>0$ such that every $\delta$-limit-pseudo orbit $\xi=(x^{(i)}) _{i\ge0}$ of $\sigma|_Y$ is $\epsilon$-limit shadowed by $x=(x_1^{(n-1)})_{n\ge1}\in Y$.
\end{rem}

The following example is taken from \cite{K2}. Let $\sigma\colon[-1,1]^\mathbb{N}\to[-1,1]^\mathbb{N}$ be the shift map and let $D$ be the metric on $[-1,1]^\mathbb{N}$ defined by
\[
D(x,y)=\sup_{n\ge1}2^{-n}|x_n-y_n|
\]
for all $x=(x_n)_{n\ge1},y=(y_n)_{n\ge1}\in[-1,1]^\mathbb{N}$. Let $s=(s_k)_{k\ge 1}$ be a sequence of numbers with $0<s_1<s_2<\cdots$ and $\lim_{k\to\infty}s_k=1$. Put
\[
S=\{-1,1\}\cup\{-s_k\colon k\ge1\}\cup\{s_k\colon k\ge1\},
\]
a closed subset of $[-1,1]$.

We define a closed $\sigma$-invariant subset $X$ of $S^\mathbb{N}$ by for any $x=(x_n)_{n\ge1}\in S^\mathbb{N}$, $x\in X$ if and only if the following properties hold
\begin{itemize}
\item $|x_1|\le|x_2|\le\cdots$,
\item for any $n\ge1$ and $k\ge1$, if $x_n=s_k$, then $x_{n+j}=-s_k$ for all $1\le j\le k$,
\item for every $n\ge1$, if $x_n=1$, then $x_{n+j}=-1$ for all $j\ge1$.
\end{itemize}

Let $f=\sigma|_X\colon X\to X$, $X_k=X\cap\{-s_k,s_k\}^\mathbb{N}$, $k\ge1$, and let
\[
X_\infty=X\cap\{-1,1\}^\mathbb{N}=\{(-1)^\infty,1(-1)^\infty\}\cup\{(-1)^m1(-1)^\infty\colon m\ge1\}.
\]
Note that $X_k$, $k\ge1$, are mixing subshifts of finite type. Let
\[
Y_k=X\cap\{-s_1,-s_2,\dots,-s_k,s_1,s_2,\dots,s_k\}^\mathbb{N},
\]
$k\ge1$. We have $Y_1\subset Y_2\subset\cdots$ and
\[
X=\overline{\bigcup_{k\ge1}Y_k}.
\]
Note that $Y_k$, $k\ge1$, are closed $f$-invariant subsets of $X$ and subshifts of finite type.

Fix a sequence $(\epsilon_k)_{k\ge1}$ of numbers with $0<\epsilon_1>\epsilon_2>\cdots$ and $\lim_{k\to\infty}\epsilon_k=0$. There are sequences $(s_k)_{k\ge1}$, $(\Delta_k)_{k\ge1}$, $(\delta_k)_{k\ge1}$ of numbers such that the following conditions are satisfied
\begin{itemize}
\item $0<s_1<s_2<\cdots$ and $\lim_{k\to\infty}s_k=1$,
\item $0<\Delta_1>\Delta_2>\cdots$ and $\lim_{k\to\infty}\Delta_k=0$,
\item $0<\delta_1>\delta_2>\cdots$ and $\lim_{k\to\infty}\delta_k=0$,
\item for every $k\ge1$, there is a continuous map $\phi_k\colon X\to Y_k$ such that
\begin{itemize}
\item $D(a,\phi_k(a))\le\Delta_k$ for all $a\in X$,
\item $\phi_k(b)=b$ for all $b\in Y_k$,
\item for any $\delta_k$-limit-pseudo orbit $\xi=(x^{(i)})_{i\ge0}$ of $f$, $\xi$ is $\epsilon_k$-limit shadowed by $x=(x_1^{(n-1)})_{n\ge1}$, and $\xi'=(y^{(i)})_{i\ge0}=(\phi_k(x^{(i)}))_{i\ge0}$ is $\epsilon_k$-limit shadowed by $y=(y_1^{(n-1)})_{n\ge1}\in Y_k$.
\end{itemize}
\end{itemize}

\begin{rem}
\normalfont
An explicit construction of $\phi_k\colon X\to Y_k$, $k\ge1$, can be found in \cite{K2}.
\end{rem}

We shall prove the following.

\begin{claim}
$f\colon X\to X$ satisfies the s-limit shadowing property.
\end{claim}

Let $k\ge1$ and let $\xi=(x^{(i)})_{i\ge0}$ be a $\delta_k$-limit-pseudo orbit of $f$ as above. We have
\[
\lim_{i\to\infty}D(x^{(i)},X_l)=0
\]
for some $l\in\{1,2,\dots\}\cup\{\infty\}$.

(1) If $1\le l\le k$, then $X_l\subset Y_k$. By
\[
D(f^i(y),x^{(i)})\le D(f^i(y),y^{(i)})+D(x^{(i)},y^{(i)})=D(f^i(y),y^{(i)})+D(x^{(i)},\phi_k(x^{(i)})),
\]
$i\ge0$, we obtain $D(f^i(y),x^{(i)})\le\epsilon_k+\Delta_k$ for all $i\ge0$, and
\[
\lim_{i\to\infty}D(f^i(y),x^{(i)})=0,
\]
thus $\xi$ is $(\epsilon_k+\Delta_k)$-limit shadowed by $y$.

(2) If $k<l<\infty$, then by taking $z^{(i)}\in X_l$, $i\ge0$, with $D(x^{(i)},z^{(i)})=D(x^{(i)},X_l)$ for all $i\ge0$, we obtain
\[
\lim_{i\to\infty}D(x^{(i)},z^{(i)})=0.
\]
Note that $(z^{(i)})_{i\ge0}$ is a limit-pseudo orbit of $\sigma|_{X_l}$ and so limit shadowed by
\[
(z_1^{(n-1)})_{n\ge1}\in\{-s_l,s_l\}^\mathbb{N}
\]
(with respect to $\sigma$). Since $X_l$ is a subshift of finite type, we have $(z_1^{(N+n-1)})_{n\ge1}\in X_l$ for all sufficiently large $N\ge0$. There is $N\ge0$ such that the following conditions are satisfied
\begin{itemize}
\item $(z_1^{(N+n-1)})_{n\ge1}\in X_l$,
\item $D(x^{(N+i)},z^{(N+i)})\le\epsilon_k$ and so $|x_1^{(N+i)}-z_1^{(N+i)}|\le2\epsilon_k$ for all $i\ge0$,
\item
\[
w=(y_1^{(0)},y_1^{(1)},\dots,y_1^{(N-1)},z_1^{(N)},z_1^{(N+1)},z_1^{(N+2)},\dots)\in X.
\]
\end{itemize}
Note that
\[
x=(x_1^{(0)},x_1^{(1)},\dots,x_1^{(N-1)},x_1^{(N)},x_1^{(N+1)},x_1^{(N+2)},\dots).
\]
By
\[
|x_1^{(i)}-y_1^{(i)}|\le2D(x^{(i)},y^{(i)})=2D(x^{(i)},\phi_k(x^{(i)}))\le2\Delta_k,
\]
$0\le i\le N-1$, we obtain $D(\sigma^i(x),\sigma^i(w))\le\max\{\Delta_k,\epsilon_k\}$ for all $i\ge0$.
By
\[
\lim_{i\to\infty}D(x^{(N+i)},z^{(N+i)})=0
\]
and so
\[
\lim_{i\to\infty}|x_1^{(N+i)}-z_1^{(N+i)}|=0,
\]
we obtain
\[
\lim_{i\to\infty}D(\sigma^i(x),\sigma^i(w))=0.
\]
Since $\xi$ is $\epsilon_k$-limit shadowed by $x$, $\xi$ is $\max\{\epsilon_k+\Delta_k,2\epsilon_k\}$-limit shadowed by $w$.

(3) If $l=\infty$, then by taking $z^{(i)}\in X_\infty$, $i\ge0$, with $D(x^{(i)},z^{(i)})=D(x^{(i)},X_\infty)$ for all $i\ge0$, we obtain
\[
\lim_{i\to\infty}D(x^{(i)},z^{(i)})=0.
\]
Note that $(z^{(i)})_{i\ge0}$ is a limit-pseudo orbit of $\sigma|_{X_\infty}$ and so limit shadowed by $z=(z_1^{(n-1)})_{n\ge1}\in\{-1,1\}^\mathbb{N}$ (with respect to $\sigma$). Since
\[
\lim_{i\to\infty}D(\sigma^i(z),X_\infty)=\lim_{i\to\infty}D(\sigma^i(z),z^{(i)})=0,
\]
we have $|\{n\ge1\colon z_1^{(n-1)}=1\}|<\infty$ or, letting $\{n\ge1\colon z_1^{(n-1)}=1\}=\{n_l\colon l\ge1\}$ and $1\le n_1<n_2<\cdots$,
\[
\lim_{l\to\infty}(n_{l+1}-n_l)=\infty.
\]
There are $N\ge0$ and $w_{N+i}\in\{s_j\colon j\ge k+1\}$, $i\ge0$, such that the following conditions are satisfied
\begin{itemize}
\item $D(x^{(N+i)},z^{(N+i)})\le\epsilon_k$ and so $|x_1^{(N+i)}-z_1^{(N+i)}|\le2\epsilon_k$ for all $i\ge0$,
\item
\[
w=(y_1^{(0)},y_1^{(1)},\dots,y_1^{(N-1)},w_N,w_{N+1},w_{N+2},\dots)\in X,
\]
\item $|z_1^{(N+i)}-w_{N+i}|\le2\epsilon_k$ for all $i\ge0$, and
\[
\lim_{i\to\infty}|z_1^{(N+i)}-w_{N+i}|=0.
\]
\end{itemize}
Note that
\[
x=(x_1^{(0)},x_1^{(1)},\dots,x_1^{(N-1)},x_1^{(N)},x_1^{(N+1)},x_1^{(N+2)},\dots).
\]
By
\[
|x_1^{(i)}-y_1^{(i)}|\le2D(x^{(i)},y^{(i)})=2D(x^{(i)},\phi_k(x^{(i)}))\le2\Delta_k,
\]
$0\le i\le N-1$, and
\[
|x_1^{(N+i)}-w_{N+i}|\le|x_1^{(N+i)}-z_1^{(N+i)}|+|z_1^{(N+i)}-w_{N+i}|\le4\epsilon_k,
\]
$i\ge0$, we obtain $D(\sigma^i(x),\sigma^i(w))\le\max\{\Delta_k,2\epsilon_k\}$ for all $i\ge0$.
By
\[
\lim_{i\to\infty}D(x^{(N+i)},z^{(N+i)})=0
\]
and so
\[
\lim_{i\to\infty}|x_1^{(N+i)}-z_1^{(N+i)}|=0,
\]
we obtain
\[
\lim_{i\to\infty}|x_1^{(N+i)}-w_{N+i}|=0,
\]
thus
\[
\lim_{i\to\infty}D(\sigma^i(x),\sigma^i(w))=0.
\]
Since $\xi$ is $\epsilon_k$-limit shadowed by $x$, $\xi$ is $\max\{\epsilon_k+\Delta_k,3\epsilon_k\}$-limit shadowed by $w$.

Since $k\ge1$ and $\xi$ are arbitrary, we conclude that $f$ satisfies the s-limit shadowing property. 

\begin{rem}
\normalfont
Given any continuous map $f\colon X\to X$, we say that a closed $f$-invariant subset $S$ of $X$ is {\em chain stable} if for any $\epsilon>0$, there is $\delta>0$ such that every $\delta$-chain $(x_i)_{i=0}^k$ of $f$ with $x_0\in S$ satisfies $d(x_i,S)\le\epsilon$ for all $0\le i\le k$. Following \cite{AHK}, we say that $C\in\mathcal{C}(f)$ is {\em terminal} if $C$ is chain stable. We denote by $\mathcal{C}_{\rm ter}(f)$ the set of terminal chain components for $f$. In \cite{K5}, it is shown that if $f$ has the shadowing property, then the union of the basins of terminal chain components for $f$, i.e.,
\[
V(f)=\bigsqcup_{C\in\mathcal{C}_{\rm ter}(f)} W^s(C)
\]
is a dense $G_\delta$-subset of $X$.

As shown in \cite{K2}, the above $f\colon X\to X$ satisfies
\[
\mathcal{C}(f)=\{X_k\colon k\ge1\}\cup\{X_\infty\}.
\]
We easily see that $\mathcal{C}_{\rm ter}(f)=\{X_\infty\}$. Since $f$ has the s-limit shadowing property and so the shadowing property, it follows that
\[
V(f)=W^s(X_\infty)
\]
is a dense $G_\delta$-subset of $X$.
\end{rem}

\begin{rem}
\normalfont
For any continuous map $f\colon X\to X$, let
\[
\mathcal{C}_{\rm sp}(f)=\{C\in\mathcal{C}(f)\colon\text{$f|_C\colon C\to C$ has the shadowing property}\}.
\]
We regard the quotient space
\[
\mathcal{C}(f)=CR(f)/{\leftrightarrow}
\]
as a space of chain components for $f$. In \cite{K3}, it is shown that if $f$ has the s-limit shadowing property, then
\[
\mathcal{C}(f)=\overline{\mathcal{C}_{\rm sp}(f)}.
\]
The above $f\colon X\to X$ satisfies $\mathcal{C}_{\rm sp}(f)=\{X_k\colon k\ge1\}$ and $\mathcal{C}(f)\setminus\mathcal{C}_{\rm sp}(f)=\{X_\infty\}$. 
\end{rem}

\begin{rem}
\normalfont
For every Furstenberg family $\mathcal{F}$, we define its dual family $\mathcal{F}^\ast$ by
\[
\mathcal{F}^\ast=\{A\subset\mathbb{N}_0\colon A\cap B\ne\emptyset\:\:\text{for all $B\in\mathcal{F}$}\},
\]
which is also a Furstenberg family. For integers $p\ge0$ and $m\ge1$, let
\[
\langle p,m\rangle=\{p+qm\colon q\ge0\}=\{p, p+m,p+2m,\dots\}.
\]
We define a Furstenberg family $\mathcal{F}_{\rm iap}$ by
\[
\mathcal{F}_{\rm iap}=\{A\subset\mathbb{N}_0\colon\langle p,m\rangle\subset A\:\:\text{for some $p\ge0$ and some $m\ge1$}\}.
\]
For any continuous map $g\colon Y\to Y$, where $Y$ is a compact metric space, and any open subset $U$ of $Y$, we have
\[
\{y\in Y\colon\{i\in\mathbb{N}_0\colon g^i(y)\in U\}\in\mathcal{F}_{\rm iap}^\ast\}=\bigcap_{p\ge0}\bigcap_{m\ge1}\bigcup_{q\ge 0}\{y\in Y\colon g^{p+qm}(y)\in U\},
\]
which is a $G_\delta$-subset of $Y$. It follows that $\mathcal{F}_{\rm iap}^\ast$ is compatible with all $(Y,g)$. We easily see that $\mathcal{F}_{\rm iap}^\ast$ is full and translation invariant.

Since $(-1)^\infty\in X_\infty$ is a fixed point for $f$, the above $f\colon X\to X$ satisfies
\[
\mathcal{D}(X_\infty)=\{X_\infty\}
\]
and so
\[
W^s(X_\infty)=V^s(X_\infty).
\]
Since
\[
\Delta_n=\inf_{D\in\mathcal{D}(X_\infty)}\sup_{x_1,x_2,\dots,x_n\in D}\min_{1\le j<k\le n}d(x_j,x_k)=\sup_{x_1,x_2,\dots,x_n\in X_\infty}\min_{1\le j<k\le n}d(x_j,x_k)>0,
\]
$n\ge2$, by Theorem 1.2 of \cite{K4}, we see that
\[
\{(x_1,x_2,\dots,x_n)\in[V^s(X_\infty)]^n\colon S_f(x_1,x_2,\dots,x_n;\delta_n)\in\mathcal{F}_{\rm iap}^\ast\}
\]
is non-empty for all $n\ge2$ for some $\delta_n>0$. Let $\mathcal{G}$ be a full Furstenberg family compatible with
\[
(X^n,f^{\times n})
\]
for all $n\ge2$. By Theorem 1.4 in Section 1, we conclude that $V^s(X_\infty)$ is generic $(\mathcal{F}_{\rm iap}^\ast,\mathcal{G})$-$n$-$r_n$-chaotic for $f$ for all $n\ge2$ and $0<r_n<\delta_n$.
\end{rem}

\end{document}